\documentclass{amsart}
\usepackage{graphicx}
\usepackage{amsmath}
\usepackage{amssymb}
\usepackage{color}
\usepackage{mathdots}
\usepackage{amsfonts}
\usepackage{mathrsfs}
\usepackage{mathtools}

\newtheorem{thm}{Theorem}[section]
\newtheorem{cor}[thm]{Corollary}

\newtheorem{lem}[thm]{Lemma}

\theoremstyle{definition}
\newtheorem{defn}[thm]{Definition}
\theoremstyle{remark}
\newtheorem{rem}[thm]{Remark}

\numberwithin{equation}{section}
\newcommand{\re}{{\rm Re}}

\newcommand{\HH}{\mathbb{H}}

\newcommand{\cS}{\mathcal{S}}

\newcommand{\cB}{\mathcal{B}}

\newcommand{\cN}{\mathcal{N}}

\newcommand{\cH}{\mathcal{H}}

\newcommand{\RR}{\mathbb{R}}
\newcommand{\BB}{\mathbb{B}}
\newcommand{\CC}{\mathbb{C}}

\newcommand{\cD}{\mathcal{D}}
\newcommand{\cL}{\mathcal{L}}

\newcommand{\Ran}{{\rm Ran}}
\newcommand{\Ker}{{\rm Ker}}

\begin{document}
\title[The spectral theorem  based on the $S$-spectrum]{The spectral theorem for quaternionic unbounded normal operators based on the $S$-spectrum} %
\author[D. Alpay]{Daniel Alpay}
\address{(DA) Department of Mathematics\\
Ben-Gurion University of the Negev\\
Beer-Sheva 84105 Israel} \email{dany@math.bgu.ac.il}

\author[F. Colombo]{Fabrizio Colombo}
\address{(FC) Politecnico di
Milano\\Dipartimento di Matematica\\Via E. Bonardi, 9\\20133
Milano, Italy}
\email{fabrizio.colombo@polimi.it}

\author[D. P. Kimsey]{David P. Kimsey}
\address{(DPK)Department of Mathematics\\
Ben-Gurion University of the Negev\\
Beer-Sheva 84105 Israel}
\email{dpkimsey@gmail.com}

\subjclass[2010]{MSC: 35P05, 47B32, 47S10}

\thanks{D. Alpay thanks the Earl Katz family for endowing the chair
which supported his research. D. P. Kimsey gratefully acknowledges the support of a Kreitman postdoctoral fellowship. F. Colombo
acknowledges the Center for Advanced Studies of the Mathematical
Department of the Ben-Gurion University of the Negev for the
support and the kind hospitality during the period in which part
of this paper has been written.}

\date{\today}%

\begin{abstract} In this paper we prove the spectral theorem for quaternionic unbounded normal operators using the notion of $S$-spectrum.
The proof technique consists of first establishing a spectral theorem for quaternionic bounded normal operators and then using a transformation which maps a quaternionic unbounded normal operator
to a quaternionic bounded normal operator. With this paper we complete the foundation  of spectral analysis of  quaternionic operators.
The $S$-spectrum has been introduced  to define the  quaternionic functional calculus but it turns out to be the correct object also for the spectral theorem for quaternionic normal operators.
The fact that the correct notion of spectrum for quaternionic operators was not previously known has been one of the main obstructions to fully understanding the
spectral theorem in this setting.
A prime motivation for studying the spectral theorem for quaternionic unbounded normal operators is
given by the subclass of unbounded anti-self adjoint quaternionic operators which play a crucial role in the quaternionic quantum mechanics.
\end{abstract}

\maketitle

\section{Introduction}
\label{sec:intro}
\setcounter{equation}{0}

In the recent paper \cite{acks2} a spectral theorem for quaternionic unitary operators based on the $S$-spectrum was proved using an extension of Herglotz's theorem to the quaternions.
In this paper, inspired by  \cite{acks2}, we treat the more general case of unbounded normal quaternionic operators.

The interest in spectral theory for quaternionic operators is motivated by the celebrated paper of Birkhoff and von Neumann, see \cite{BvN},
who showed that Schr\"odinger equation can be written only in the complex or quaternionic setting.
Several authors, have given important contributions to the development of the quaternionic version of quantum mechanics, see \cite{adler, 12, 14, 21},
but a correct notion of spectrum for quaternionic operators was still missing until the introduction of the $S$-spectrum, see, e.g., \cite{CSS}.
As it is well known, in the classical formulation of quantum mechanics the spectral theory of unbounded self-adjoint operators play a crucial role.
In the fundamental paper \cite{vonSPECT}, von Neumann used the spectral theorem for unitary operators to prove the spectral theorem for unbounded self-adjoint operators.
In quaternionic quantum mechanics the most important quaternionic operators are unbounded anti self-adjoint operators;
these operators are a particular case of quaternionic unbounded normal operators treated in this paper.

Our strategy to prove the spectral theorem is as follows: first we deduce the spectral theorem for quaternionic
bounded normal operators. The proof is based on a continuous functional calculus defined in \cite{GMP}
and a classical version of the Riesz representation theorem. After we establish a spectral theorem for quaternionic bounded normal operators,
we deduce a spectral theorem for quaternionic unbounded normal operators from the bounded case and from a suitable transformation.

With the quaternionic spectral theorem based on the $S$-spectrum we complete the foundation of the quaternionic spectral theory that started
some years ago with the introduction of the $S$-functional calculus.
In fact using the notion of slice hyperholomorphic functions, see \cite{CSS}, and the $S$-spectrum it is possible to define the quaternionic version of the Riesz-Dunford functional calculus which we now call quaternionic functional calculus or $S$-functional calculus.

We give a quick explanation of the reason why a consistent spectral theory for quaternionic operators is not so obvious.
For simplicity consider a complex bounded operator $A: \mathcal{X} \to \mathcal{X}$ on a complex Banach space $\mathcal{X}$. The spectrum of $A$ is defined as
$$
\sigma(A)=\{\lambda\in \mathbb{C}\ :\ \lambda I_{\mathcal{X}} -A \ {\rm  is \ not \ invertible } \ B(\mathcal{X})\}
$$
where $B(\mathcal{X})$ denotes the Banach space of all bounded linear operators on $\mathcal{X}$.
Given a normal (bounded) linear operator $T$ on a complex Hilbert space,  in the spectral theorem
$$
T=\int_{\sigma(T)}\lambda \, dE(\lambda)
$$
the unique
spectral measure $E(\lambda)$ associated to $T$ is supported on $\sigma(T)$, see, e.g., \cite{ds2}.
The above notion of spectrum also appears in the Riesz-Dunford functional calculus, see \cite{ds}, which is based on the Cauchy formula of holomorphic functions
in which the Cauchy kernel is replaced by the resolvent operator
$(\lambda I_{\mathcal{X}} -T)^{-1}$. Taking a holomorphic function $h$ defined on an open set that contains the spectrum, we can use the Cauchy formula to define the linear operator $h(T)$.

From a historical view point a first attempt to generalize the classical notion of spectrum
to quaternionic linear operators was to readapt the definition. To see the inconsistencies that occur consider a right linear quaternionic operator
$T:\mathcal{V}\to \mathcal{V}$ acting on a quaternionic two-sided Banach space $\mathcal{V}$.
The symbol $\mathcal{B}(\mathcal{V})$ denotes the Banach space of all bounded right linear quaternionic operators on $\mathcal{V}$.
The left spectrum $\sigma_L(T)$ of $T$ is related to the left-resolvent operator
$(s I_{\mathcal{V}}-T)^{-1}$, i.e.,
$$
\sigma_L(T)=\{\text{$s\in \mathbb{H}$ {\rm:} $s I_{\mathcal{V}}-T$ is not invertible in $\mathcal{B}(\mathcal{V})$}\},
$$
where
$$(s I_{\mathcal{V}})(v) = sv, \quad v \in \mathcal{V}.$$

The right spectrum $\sigma_R(T)$ of $T$ is associated with the  right eigenvalue problem, i.e., the search for nonzero vectors $v$
satisfying $T(v)=vs$.
Observe that the operator $I_{ \mathcal{V} } s -T$ associated with the right eigenvalue problem is not linear. Consequently, it is not clear what the resolvent operator ought to be.
The quaternionic left-resolvent operator $(s I_{\mathcal{V}}-T)^{-1}$, as far as we know, is not hyperholomorphic in any sense. Consequently, the left-resolvent operator is not
useful to define a hyperholomorphic quaternionic functional calculus. When we consider the right spectrum we just have the notion of eigenvalues.
The above discussion shows that there is a problem in adapting the classical notion of spectrum to either the left or right quaternionic spectrum.

As we shall see, relative to obtaining a spectral theorem for quaternionic normal operators, the appropriate notion of spectrum is a new notion of spectrum which is as follows.
The $S$-spectrum, see \cite{CSS}, is defined as
$$
\sigma_S(T)=\{\text{$s\in \mathbb{H}$ {\rm :} $T^2-2 {\rm Re}(s)T+|s|^2 I_{\mathcal{V}}$ is not invertible in $\mathcal{B}(\mathcal{V})$}\},
$$
where $s=s_0+s_1e_1+s_2e_2+s_3e_3$ is a quaternion,
$\{1, e_1, e_2, e_3 \}$ is the standard basis of the quaternions, ${\rm Re}(s)=s_0$ is the real part and the norm $|s| = \sqrt{s_0^2+s_1^2+s_2^2+s_3^2}$.

We are now ready to illustrate our main result, the spectral theorem for normal quaternionic operators.
We limit the  discussion to the case of bounded normal operators but the theorem holds also for unbounded operators, see Theorem \ref{thm:August27t1}.

Consider the complex plane $\mathbb{C}_j:=\mathbb{R}+j\mathbb{R}$, for $j\in \mathbb{S}$, where $\mathbb{S}$ is the unit sphere of purely imaginary quaternions. Let $\CC_j^+$ denote all $p \in \CC_j$ with ${\rm Im}( p) \geq 0$.
Observe that $\mathbb{C}_j$ can be identified with a complex plane since $j^2=-1$ for every $j\in \mathbb{S}$.
If $T$ be a (bounded) right linear normal operator on a quaternionic Hilbert space, then for $j\in \mathbb{S}$
there is a unique spectral measure $E$ and a Hilbert basis $\cN_j$ of $\cH$ so that
$$
T=\int_{\sigma_S(T)\cap \mathbb{C}_j^+}\, p\, dE(p).
$$
To show the deep difference between complex spectral theory and the quaternionic spectral theory, we recall the quaternionic version of the Riesz-Dunford functional calculus,
which suggests the notion of $S$-spectrum, see \cite{formulation, JGA}.
 This calculus involves two resolvent operators, namely a left and right $S$-resolvent operators given by
\begin{equation}
S_L^{-1}(s,T):=-(T^2-2{\rm Re}(s) T+|s|^2I_{\mathcal{V}})^{-1}(T-\overline{s}I_{\mathcal{V}}),\ \ \ s \in \rho_S(T)
\end{equation}
and
\begin{equation}
S_R^{-1}(s,T):=-(T-\overline{s} I_{\mathcal{V}})(T^2-2{\rm Re}(s) T+|s|^2I_{\mathcal{V}})^{-1},\ \ \ s \in  \rho_S(T),
\end{equation}
where $T\in \mathcal{B}(\mathcal{V})$ and $\rho_S(T)=\mathbb{H}\setminus \sigma_S(T)$ is the $S$-resolvent set.
As one can see the $S$-spectrum is suggested by the $S$-resolvent operators.

Let $\Omega\subset \mathbb{H}$ be a suitable domain that contains the $S$-spectrum of $T$.
We define the quaternionic functional calculus for left slice hyperholomorphic functions $f:\Omega \to \mathbb{H}$ as
\begin{equation}\label{quatinteg311def}
f(T)={{1}\over{2\pi }} \int_{\partial (\Omega\cap \mathbb{C}_j)} S_L^{-1} (s,T)\  ds_j \ f(s),
\end{equation}
where $ds_j=- ds j$;
for right slice hyperholomorphic functions, we define \begin{equation}\label{quatinteg311rightdef}
f(T)={{1}\over{2\pi }} \int_{\partial (\Omega\cap \mathbb{C}_j)} \  f(s)\ ds_j \ S_R^{-1} (s,T).
\end{equation}
These definitions are well posed since the integrals depend neither on the open set $\Omega$ nor on the complex plane
$\mathbb{C}_j$.
Moreover, the resolvent equation, see \cite{acgs}, involves both $S$-resolvent operators,
for $s$ and $p\in \rho_S(T)$ we have
\begin{align}
S_R^{-1}(s,T)S_L^{-1}(p,T)=& \; \{ (S_R^{-1}(s,T)-S_L^{-1}(p,T))p
-\overline{s}(S_R^{-1}(s,T)-S_L^{-1}(p,T)) \} \nonumber \\
\times& \; (p^2-2s_0p+|s|^2)^{-1}. \label{RLresolv}
\end{align}

Even though there are deep differences with respect to the classical resolvent equation for complex operators,
all of the results that hold for the Riesz-Dunford functional calculus also hold for the quaternionic functional calculus.
{\it We now claim that to replace the complex spectral theory with the quaternionic spectral theory we have to replace the
classical spectrum with the $S$-spectrum.}

We conclude with some final remarks. In the case $T$ is a right linear operator on a finite-dimensional Hilbert space, the $S$-spectrum of $T$ coincides with the set of right eigenvalues of $T$; in the general case of a linear operator, the point $S$-spectrum coincides with the set of right eigenvalues. In the literature the spectral theorem for quaternionic normal matrices based on the right spectrum is proved in \cite{fp}.
In the literature, there are some papers on the quaternionic spectral theorem, see, e.g., \cite{Teichmueller, 14, sc, Viswanath}. However, the notion of spectrum in the papers \cite{14, sc, Viswanath} is not made clear.
In \cite{spectcomp}, a spectral theorem based on $S$-spectrum is proved for compact normal operators on a quaternionic Hilbert space.
We point out that the $S$-resolvent operators are also used in Schur analysis in
the realization of Schur functions in the slice hyperholomorphic setting see \cite{aacks, acs1, acs2, acs3}
and \cite{MR2002b:47144, adrs} for the classical case.
In the papers \cite{perturbation, evolution, GR} the problem of the generation of
quaternionic groups and semigroups is treated using the $S$-spectrum.

The plan of the paper as follows: In Section \ref{sec:prelim} we give some preliminaries; in Section \ref{sec:funcCalc} we recall a continuous functional calculus
for bounded normal operators; in Section \ref{sec:STB} we prove the spectral theorem for bounded normal operators based on the $S$-spectrum; in Section \ref{sec:SIs} we introduce spectral integrals; finally, in Section \ref{sec:STU}
we prove the spectral theorem for unbounded normal operators based on the $S$-spectrum.

\section{Preliminaries}
\label{sec:prelim}
Let $\cH$ be a right linear quaternionic Hilbert space with an $\mathbb{H}$-valued inner product $\langle \cdot, \cdot \rangle$ which satisfies, for every $\alpha$, $\beta\in \mathbb{H}$, and $x$, $y$, $z\in \cH$, the relations:
\begin{align*}
\langle x, y \rangle =& \; \overline{ \langle y, x \rangle }. \\
\langle x, x \rangle \geq& \; 0 \quad {\rm and} \quad \| x \|^2 := \langle x, x \rangle = 0 \Longleftrightarrow x = 0. \\
\langle x \alpha + y \beta, z \rangle =& \; \langle x, z \rangle \alpha + \langle y, z \rangle \beta. \\
\langle x, y \alpha + z \beta \rangle =& \; \bar{\alpha} \langle x, y \rangle + \bar{\beta} \langle x, z \rangle.
\end{align*}
We call an operator $T: \cD(\cH) \to \cH$ {\it right linear} if
$$
T(x \alpha + y \beta) = (Tx ) \alpha + (Ty) \beta,
$$
for all $x, y$ in the domain of $T$ and $\alpha, \beta \in \mathbb{H}$. The set of right linear operators on $\cH$ will be denoted by $\cL(\cH)$. Given $T \in \cL(\cH)$, the domain of $T$ will be denoted by $\cD(T)$
and the range and kernel of $T$ will be given by
$$\Ran(T) = \{\text{$y \in \cH${\rm :} $Tx = y$ for $x \in \cD(\cH)$}\}$$
and
$$\Ker(T) = \{\text{$x \in \cD(T)${\rm :} $Tx = 0$}\},$$
respectively.
We call an operator $T \in \cL(\cH)$ {\it bounded} if
$$\| T \| := \sup_{\| x \| \leq 1} \| Tx \| < \infty.$$

In the sequel $ \cB(\cH)$ will denote the Banach space of all bounded right linear operators on $ \cH$ endowed with the natural norm.

\begin{defn}
\label{def:Dec14tt9}
An operator $T \in \cL(\cH)$ is called {\it closed} if the set $\{ (x, Tx): x \in \cH \}$ is a closed subset of $\cH\times
\cH$. Let $S$ and $T$ both belong $\cL(\cH)$. We write $S = T$ if $\cD(S) = \cD(T)$ and $Sx = Tx$ for all $x \in \cD(S) = \cD(T)$. We write $S \subseteq T$ if $\cD(S) \subseteq \cD(T)$ and $Sx = Tx$ for all $x \in \cD(S)$. Clearly, $S = T$ if and only if $S \subseteq T$ and $T \subseteq S$. An operator $T \in \cL(\cH)$ is called {\it closable} if there exists a closed operator $\overline{T} := W \in \cL(\cH)$ so that $T \subseteq W$.
\end{defn}

\begin{defn}
\label{def:Dec11rro3}
Given $T \in \cL(\cH)$ which is densely defined, we let $T^* \in \cL(\cH)$ denote the unique operator so that
$$\langle Tx, y \rangle = \langle x, T^* y \rangle, \quad x\in \cD(T).$$
The domain of $T^*$ is given by
$$\cD(T^*) = \{\text{$y \in \cH$ : there exists $z \in \cH$ with $\langle Tx, y \rangle = \langle x, z \rangle$}\}.$$
\end{defn}

\begin{thm}
\label{thm:Dec14yy1}
If $T \in \cL(\cH)$ is densely defined and $W \in \cL(\cH)$, then{\rm :}
\begin{enumerate}
\item[(i)] $T^* \in \cL(\cH)$ is closed.
\item[(ii)] ${\rm \Ran}(T)^{\perp} = {\rm Ker}(T^*)$.
\item[(iii)] If $T \subseteq W$, then $W^* \subseteq T^*$.
\end{enumerate}
\end{thm}

\begin{proof}
The proofs can completed in much the same way as the case when $\cH$ is a complex Hilbert space (see, e.g., Proposition 1.6 in \cite{Schmuedgen}).
\end{proof}

\begin{thm}
\label{thm:Dec14uu1}
If $T \in \cL(\cH)$ is densely defined, then{\rm :}
\begin{enumerate}
\item[(i)] $T$ is closable if and only if $\cD(T^*)$ is dense in $\cH$.
\item[(ii)] If $T$ is closable, then $\overline{T} = T^{**}$.
\item[(iii)] $T$ is closed if and only if $\overline{T} = T^{**}$.
\item[(iv)] If $T$ is closable and ${\rm Ker}(T) = \{ 0 \}$, then $T^{-1}$ is closable if and only if ${\rm Ker}(\overline{T}) = \{ 0 \}$. Moreover,
$$(\overline{T})^{-1} = \overline{T^{-1}}.$$
\end{enumerate}
\end{thm}

\begin{proof}
The proofs can completed in much the same way as the case when $\cH$ is a complex Hilbert space (see, e.g., Theorem 1.8 in \cite{Schmuedgen}).
\end{proof}

\begin{defn}
\label{def:Aug24y1}
Let $T \in \cL(\cH)$. We call $T$ {\it normal} if $T$ is densely defined, $T$ is closed and $TT^* = T^* T$.
\end{defn}

\begin{lem}
\label{lem:Aug25yq1}
Let $T \in \cL(\cH)$ be normal. If $S \in \cL(\cH)$ so that $T \subseteq S$ and $\cD(S) \subseteq D(S^*)$, then $S = T$.
\end{lem}

\begin{proof}
If $T \subseteq S$, then $S^* \subseteq T^*$ and hence
$$\cD(T) \subseteq \cD(S) \subseteq \cD(S^*) \subseteq \cD(T^*) = \cD(T),$$
i.e., $\cD(S) = \cD(T)$. Therefore, $S = T$.
\end{proof}

\begin{defn}
\label{def:Sept15t1}
Let $T \in \cL(\cH)$. We call $T$ {\it self-adjoint}, {\it anti self-adjoint} and {\it unitary} if $T = T^*$, $T = - T^*$ and $T T^* = T^* T = I_{\cH}$, respectively.
\end{defn}

\begin{defn}
\label{def:Dec14j1}
Let $T \in \cL(\cH)$ be densely defined and let $\mathcal{R}_s(T): \cD(T^2) \to \cH$ be given by
$$
\mathcal{R}_s(T)x  =\{ T^2 - 2 {\rm Re}(s) T + |s|^2 I_{\mathcal{H}} \}x, \quad x \in \mathcal{D}(T^2).
$$
The $S$-resolvent set of $T$ is defined as follows
\begin{align*}
\rho_S(T)
=& \; \{\text{$s\in \mathbb{H}$ {\rm :} $\Ker(\mathcal{R}_s(T))=\{0\}$, $\Ran(\mathcal{R}_s(T))$ is dense in $\cH$ and} \\
& \; \text{$\mathcal{R}_s(T)^{-1} \in \cB(\cH)$}\}.
\end{align*}
\end{defn}

We recall some properties, see \cite{CSS}, of the $S$-spectrum. The $S$-spectrum satisfies
$$
\sigma_S(T)=\mathbb{H}\setminus \rho_S(T).
$$

\begin{thm}
\label{thm:Sept15uu1}
Let $T \in \mathcal{B}(\cH)$. Then the $S$-spectrum is a compact non-empty subset of $\HH$ and
\begin{equation}
\label{eq:Dec5tt1}
\sigma_S(T) \subseteq \{ p \in \HH: 0 \leq |p| \leq \| T \| \}.
\end{equation}
\end{thm}

\begin{proof}
See Theorem 3.2.6 in \cite{CSS}.
\end{proof}

\begin{thm}
\label{thm:Sept15uu1b}
Let $T \in \cL(\cH)$ be densely defined.  If $p=p_0+ip_1\in  \sigma_S(T)$ for $i\in \mathbb{S}$,
then $p_0+jp_1\in  \sigma_S(T)$ for all $j\in \mathbb{S}$.
\end{thm}

\begin{proof}
The proof of the assertion follows directly from the definition of the $S$-spectrum.
If $s \in \sigma_S(T)$, then it follows immediately from the definition of $\sigma_S(T)$ that all the quaternions with the same real part and the same modulus belong to the $S$-spectrum of $T$.
\end{proof}

\begin{thm}
\label{thm:Sept15yz1}
Let $T \in \cL(\cH)$. The following statements hold{\rm :}
\begin{enumerate}
\item[(i)] If $T$ is positive, then $\sigma_S(T) \subseteq [0, \infty)$. If, in particular, $T \in \cB(\cH)$ is positive, then
$$\sigma_S(T) \subseteq [0, \| T \|].$$
\item[(ii)] If $T$ is self-adjoint, then $\sigma_S(T) \subseteq \RR$. If, in particular, $T \in \cB(\cH)$ is self-adjoint, then
$$\sigma_S(T) \subseteq [ -\| T \|, \| T \| ].$$
\item[(iii)] If $T$ is anti self-adjoint, then $\sigma_S(T) \subseteq \{\text{$p \in \HH${\rm :} $\re \hspace{0.5mm} p = 0$}\}$. If, in particular, $T \in \cB(\cH)$ is anti self-adjoint, then
$$\sigma_S(T)   \subseteq \{\text{$p \in \HH${\rm :} $\re \hspace{0.5mm} p = 0$ {\it and} $|p| \leq \| T \|$}\}.$$
\item[(iv)] If $T$ is unitary, then $\sigma_S(T) \subseteq \mathbb{S}$.
\end{enumerate}
\end{thm}

\begin{proof}
If $T \in \cL(\cH)$, then the containments illustrated in (i)-(iii) follow readily from the definition of $\sigma_S(T)$. If $T \in \cB(\cH)$, then the containments illustrated in (i)-(iv) follow readily from \eqref{eq:Dec5tt1}.
\end{proof}

We will also need the following version of the Riesz representation theorem.

\begin{thm}
\label{thm:July29u1}
Let $X$ be a compact Hausdorff space and $\mathscr{C}(X, \RR)$ denote the normed space of real-valued continuous functions on $X$ together with the supremum norm $\| \cdot \|_{\infty}$. Corresponding to any bounded positive linear functional $\psi: \mathscr{C}(X, \RR) \to \RR$ there exists a unique positive Borel measure $\mu$ on $X$ such that
\begin{equation}
\label{eq:July30y1}
\psi(f) = \int_X f(p) d\mu(p) \quad {\rm for} \quad {\rm all} \quad f \in \mathscr{C}(X, \RR).
\end{equation}
\end{thm}

\begin{proof}
The assertion is a special case of Theorem D in Section 56 of \cite{Halmos}.
\end{proof}

\section{A functional Calculus for bounded normal operators}
\label{sec:funcCalc}
Let $\cH$ be a right quaternionic Hilbert space and $\cB(\cH)$ denote the set of all bounded right linear operators on $\cH$.
In the recent paper \cite{GMP}, Ghiloni, Moretti and Perotti established the existence of several functional calculi for a quaternionic bounded normal operator. Before introducing the functional calculus for bounded normal operators, we first need some notation and results.

\begin{defn}
\label{def:Sept22ut1}
Fix a Hilbert basis $\mathcal{N}$ of a quaternionic Hilbert space $\cH$. The  {\it left scalar multiplication} $L_p$ of $\cH$ induced by $\cN$ is the map
$$(p, x) \in \mathbb{H} \times \cH \mapsto p x \in \cH$$
given by
$$p x := \sum_{y \in \cN} y p \langle x, y \rangle.$$
\end{defn}

\begin{lem}[Statement (a) of Proposition 3.8 in \cite{GMP}]
\label{lem:Sept22un1}
Let $\cH$ be a quaternionic Hilbert space. If $J \in \cB(\cH)$ is an anti self-adjoint and unitary operator, then corresponding to any fixed $j \in \mathbb{S}$, there exists a left-scalar multiplication $L_p$ so that
$$J = L_j.$$
\end{lem}

\begin{thm}
\label{thm:Sept1ub1}
Let $T \in \cB(\cH)$ be normal. Then there exist uniquely determined operators $A := (1/2)(T+T^*)$ and $B := (1/2)|T-T^*|$ which both belong to $\cB(\cH)$ and an operator $J \in \cB(\cH)$ which is uniquely determined on $\{ {\rm Ker}(T - T^*) \}^{\perp}$ so that the following properties hold{\rm :}
\begin{enumerate}
\item[(i)] T = A + J B.
\item[(ii)] $A$ is self-adjoint and $B$ is positive.
\item[(iii)] $J$ is anti self-adjoint and unitary.
\item[(iv)] $A$, $B$ and $J$ mutually commute.
\item[(v)] For any fixed $j \in \mathbb{S}$, there exists a Hilbert basis $\cN_j$ of $\cH$ with the property that $J = L_j$.
\end{enumerate}
\end{thm}

\begin{proof}
Properties (i)-(iv) appear in Theorem J on page 4 of \cite{GMP}. Property (v) follows from Lemma \ref{lem:Sept22un1}.
\end{proof}

\begin{defn}
\label{def:Sept15kj1}
Let $\Omega \subseteq \HH$. We call $\Omega$ {\it axially symmetric} if for every point $p_0+ip_1\in \Omega$ with $i\in \mathbb{S}$, then $p_0+jp_1\in \Omega$ for all $j\in \mathbb{S}$.
\end{defn}

\begin{rem}
\label{rem:Sept15pp1}
Let $T \in \cL(\cH)$. In view of Theorem \ref{thm:Sept15uu1b}, $\sigma_S(T)$ is an axially symmetric subset of $\HH$.
\end{rem}

\begin{defn}
\label{def:Sept15kj2}
Let $\cS(\Omega, \HH)$ denote the quaternionic linear space of slice continuous functions on an axially symmetric subset $\Omega$ of $\HH$, i.e., $\cS(\Omega, \HH)$ consists of functions $f: \HH \to \HH$ of the form
$$f(u+vj ) = \alpha(u,v) + j\beta(u,v), \quad j \in \mathbb{S},$$
where $\alpha$ and $\beta$ are continuous $\HH$-valued functions so that
$$\alpha(u,v) = \alpha(u,-v) \quad {\rm and} \quad \beta(u,v) = - \beta(u, -v).$$
If $\alpha$ and $\beta$ are real-valued, then we say that the continuous slice function $f$ is {\it intrinsic}. The subspace of intrinsic continuous slice functions is denoted by $\cS_{\RR}(\Omega, \HH)$. The subspace of $\CC_j$-valued functions in $\cS_{\RR}(\Omega, \HH)$ will be denoted by $\cS_{\RR}(\Omega, \CC_j)$.
\end{defn}

The following functional calculus will be useful for proving a spectral theorem for a normal operator $T \in \cB(\cH)$.

\begin{thm}[Theorem 7.4 in \cite{GMP}]
\label{thm:July30e1}
Let $T \in \cB(\cH)$ be normal. There exists a unique continuous *-homomorphism
$$\Psi_{\RR, T}: f \in \cS_{\RR} (\sigma_S(T), \HH) \mapsto f(T) \in \cB(\cH)$$
of real-Banach unital $C^*$-algebras such that{\rm :}
\begin{enumerate}
\item[(i)] $\Psi_{\RR, T}(\chi_{\sigma_S(T)}) = I_{\cH}$,
where
$$\chi_{\sigma_S(T)}(p) = \begin{cases} 1 & {\rm if} \quad p \in \sigma_S(T) \\
0 & {\rm if} \quad p \notin \sigma_S(T).
\end{cases}$$
\item[(ii)] $\Psi_{\RR, T}({\rm id}) = T$, where ${\rm id}$ denotes the inclusion map from $\sigma_S(T)$ to $\HH$.
\item[(iii)] If $J$ is as in Theorem \ref{thm:Sept1ub1}, then $J$ commutes with the normal operator $f(T)$.
\item[(iv)] If $f \in \cS_{\RR}(\sigma_S(T), \HH)$, then $\| f(T) \| = \| f \|_{\infty}$.
\item[(v)] If $f \in \cS_{\RR}(\sigma_S(T), \HH)$, then
\begin{equation}
\label{eq:Nov24t2}
\sigma_S(f(T)) = f(\sigma_S(T)).
\end{equation}
\end{enumerate}
\end{thm}

\begin{rem}
\label{rem:Dec9t1}
For the convenience of the reader, we will now outline the construction of $f(T)$ for $f \in \cS_{\RR}(\sigma_S(T), \HH)$.
Since $\sigma_S(T)$ is compact, there exist sequences of real-valued polynomials $\{ \phi_n(u,v) \}_{n=0}^{\infty}$ and $\{ \psi_n(u,v) \}_{n=0}^{\infty}$, so that
\begin{equation}
\label{eq:Sept21y1}
\text{$f_0(u,v) = \lim_{n \uparrow \infty} \phi_n(u,v)$ uniformly on $\sigma_S(T)$}
\end{equation}
and
\begin{equation}
\label{eq:Sept21y2}
\text{$f_1(u,v) = \lim_{n \uparrow \infty} \psi_n(u,v)$ uniformly on $\sigma_S(T)$},
\end{equation}
respectively.

The polynomials $\phi_n(u,v)$ and $\psi_n(u,v)$ can be constructed such that they are both slice continuous.

Since $\phi_n(u,v)$ and $\psi_n(u,v)$ have real coefficients and $A$ and $B$ are commuting self-adjoint operators it follows easily that $\phi_n(A,B) \in \cB(\cH)$ and $\psi_n(A,B) \in \cB(\cH)$ are self-adjoint. Next, we define
\begin{equation}
\label{eq:Oct19j1}
f_0(T)x := \lim_{n \uparrow \infty} \phi_n(A,B)x, \quad x \in \cH
\end{equation}
and
\begin{equation}
\label{eq:Oct19j2}
f_1(T)x := \lim_{n \uparrow \infty} \psi_n(A,B)x, \quad x \in \cH.
\end{equation}
Note that the limit in \eqref{eq:Oct19j1} exists since $\phi_n(A,B) = \phi_n(A,B)^*$ and hence
\begin{align}
\| \{ \phi_m(A,B) - \phi_n(A,B) \} x \|^2 =& \; \langle \{ \phi_m(A,B) - \phi_n(A,B) \}^2 x, x \rangle \nonumber \\
\leq& \; \|  \{ \phi_m(A,B) - \phi_n(A,B) \}^2 \| \| x \|^2 \nonumber \\
=& \; \| \phi_m - \phi_n \|^2_{\infty} \| x \|^2 \label{eq:Oct19uv1} \\
\to & \; 0, \quad {\rm as} \quad m,n \uparrow \infty, \nonumber
\end{align}
since \eqref{eq:Sept21y1} holds. Note that item (iv) in Theorem \ref{thm:July30e1} was used to obtain \eqref{eq:Oct19uv1}. The verification of the existence of the limit given in \eqref{eq:Oct19j2} is similar. The normal operator $f(T) \in \cB(\cH)$ is given by
\begin{equation}
\label{eq:Dec9rr2}
f(T) = f_0(T) + J f_1(T).
\end{equation}
\end{rem}

\begin{lem}
\label{lem:Sept21uuz1}
Fix a normal operator $T \in \cB(\cH)$. If $f = f_0 + f_1 j \in \cS_{\RR}(\sigma_S(T), \HH)$, then $f_0(T)$ and $f_1(T)$ {\rm (}given in \eqref{eq:Sept21y1} and \eqref{eq:Sept21y2}, respectively{\rm )} are self-adjoint.
\end{lem}

\begin{proof}
We claim that
\begin{equation}
\label{eq:Sept21t1}
\lim_{n \uparrow \infty} \langle \phi_n(A,B) x, y \rangle = \langle f_0(T)x, y \rangle, \quad x,y \in \cH,
\end{equation}
and
\begin{equation}
\label{eq:Sept21t2}
\lim_{n \uparrow \infty} \langle \psi_n(A,B) x, y \rangle = \langle f_1(T)x, y \rangle, \quad x,y, \in \cH.
\end{equation}
Assertion \eqref{eq:Sept21t1} follows directly from
\begin{align*}
| \langle \phi_n(A,B)x, y \rangle - \langle f_0(T)x, y \rangle | \leq& \; \| \phi_n(A,B) - f_0(T) \| \| x \| \| y \| \\
=& \; \| \phi_n - f_0 \|_{\infty} \| x\| \| y \|,
\end{align*}
where item (iv) of Theorem \ref{thm:July30e1} was used to obtain the last line. Assertion \eqref{eq:Sept21t2} is shown in much the same way. In view of \eqref{eq:Sept21t1},
\begin{align*}
\langle f_0(T)x, y \rangle =& \; \lim_{n \uparrow \infty} \langle \phi_n(A,B)x, y \rangle \\
=& \; \lim_{n \uparrow \infty} \langle x, \phi_n(A,B) y \rangle \\
=& \; \langle x, f_0(T)y \rangle, \quad x,y \in \cH.
\end{align*}
Thus, $f_0(T)$ is self-adjoint. The fact that $f_1(T)$ is self-adjoint can be completed in much the same way using \eqref{eq:Sept21t2}.
\end{proof}

\section{The spectral theorem for bounded normal operators based on the $S$-spectrum}
\label{sec:STB}
In this section we shall consider normal operators $T$ which are bounded, i.e., $T \in \cB(\cH)$. We will generate a spectral theorem based on the $S$-spectrum using Theorems \ref{thm:July29u1} and \ref{thm:July30e1}. This approach is analogous to a well-known approach in the classical case, i.e., when $\cH$ is a complex Hilbert space. See, e.g., the book of Lax \cite{Lax} for details.

Fix a normal operator $T \in \cB(\cH)$ and $j \in \mathbb{S}$. By Theorem \ref{thm:Sept1ub1}, there exist commuting operators $A := (1/2)(T+T^*)$ and $B := (1/2) |T - T^*|$, where $|W| = (W^*W)^{1/2}$ for $W \in \cB(\cH)$, and $J$ which all belong to $\cB(\cH)$ so that $T = A + JB$ and $A$ and $B$ are uniquely determined by $T$. Moreover, in view of item (v) of Theorem \ref{thm:Sept1ub1} we have the existence of a Hilbert basis $\cN_j$ of $\cH$ so that if $J$ is written with respect to $\cN_j$,
\begin{equation}
\label{eq:Sept22tr1}
J = L_j.
\end{equation}
In what follows we shall assume, without loss of generality, that $A$, $B$, and $J$ (and consequently $T$) are written with respect to $\cN_j$.

\begin{lem}
\label{lem:Nov29q1}
Let $\mathscr{C}(\Omega_j^+, \RR)$ denote the set of real-valued continuous functions on $\Omega^+_j = \sigma_S(T)\cap \CC^+_j$ and $\cS_{\RR}(\Omega_j, \RR)$ denote the set of real-valued functions in $\cS_{\RR}(\Omega_j, \HH)$, where $\Omega_j = \sigma_S(T) \cap \CC_j$. There exists a bijection between $\mathscr{C}(\Omega_j^+, \RR)$ and $\cS_{\RR}(\Omega_j, \RR)$. Moreover, there exists a bijection between $\mathscr{C}(\Omega_j^+, \RR)$ and purely imaginary functions in $\cS_{\RR}(\Omega_j, \HH)$.
\end{lem}

\begin{proof}
If $g \in \mathscr{C}(\Omega_j^+, \RR)$, then the function
$$\tilde{g}(u,v) = \begin{cases}
g(u,v) & {\rm if} \quad u+jv \in \Omega_j^+ \\
g(u,-v) & {\rm if} \quad u+jv \in \Omega_j^-
\end{cases}
$$
belongs to $\cS_{\RR}(\Omega_j, \RR)$. Conversely, if $f \in \cS_{\RR}(\Omega_j, \CC_j)$ is real-valued, then
$\tilde{f} = f|_{\Omega_j^+} \in \mathscr{C}(\Omega_j^+, \RR)$.

The proof of the second assertion is completed in much the same way as the first assertion.
\end{proof}

Fix $x \in \cH$ and let
$$\ell_x(g) = \langle g(T)x, x \rangle, \quad g \in \mathscr{C}(\Omega_j^+, \RR).$$
It is readily checked that $\ell_x$ is a real-valued bounded linear functional on the compact Hausdorff space $\mathscr{C}(\Omega_j^+, \RR)$. Moreover, $\ell_x$ is a positive functional. Indeed, if $f$ is a continuous nonnegative function on $\Omega_j^+$, then $g$ given by $g(u,v) = \sqrt{f(u,v)}$ also belongs to $\mathscr{C}(\Omega_j^+, \RR)$ and $g(T) = g(T)^*$. Thus,
\begin{align*}
\langle f(T)x, x \rangle =& \; \langle g(T)x, g(T)x \rangle \\
=& \; \| g(T) x\|^2 \geq 0.
\end{align*}

Theorem \ref{thm:July29u1} yields the existence of a uniquely determined positive valued measure $\mu_x$ (for a fixed $j \in \mathbb{S}$) so that
\begin{equation}
\label{eq:Sept22y4}
{\ell}_x(g) = \int_{\Omega_j^+} g(p) d\mu_x(p), \quad g \in \mathscr{C}(\Omega_j^+, \RR).
\end{equation}

In view of \eqref{eq:Sept22y4}, we may use the polarization formula
\begin{align}
4\langle Tx, y \rangle =& \; \langle T(x+y), x+y \rangle - \langle T(x-y), x-y \rangle + e_1 \langle T(x+ye_1), x+ye_1 \rangle \nonumber \\
-& \; e_1 \langle T(x-ye_1), x-ye_1 \rangle  + e_1 \langle T(x-ye_2), x-ye_2 \rangle e_3 \nonumber \\
-& \; e_1 \langle T(x + y e_2), x + y e_2 \rangle e_3 + \langle T(x+y e_3), x+y e_3 \rangle e_3 \nonumber \\
-& \; \langle T(x- y e_3), x - y e_3 \rangle e_3, \label{eq:Sept24yy1}
\end{align}
where $\{ 1, e_1, e_2, e_3 \}$ denotes the standard basis of $\mathbb{H}$,
to obtain a uniquely determined quaternion-valued measure $\mu_{x,y}$ (relative to a fixed $j \in \mathbb{S}$) so that
\begin{equation}
\label{eq:Sept18i1}
\langle g(T)x, y \rangle = \int_{\Omega_j^+} g(p) d\mu_{x,y}(p), \quad g \in \mathscr{C}(\Omega_j^+, \RR),
\end{equation}
where
\begin{align}
4\mu_{x,y} =& \; \mu_{x+y} - \mu_{x-y} + e_1 \mu_{x+y e_1} - e_1 \mu_{x-y e_1}  \label{eq:Nov29u2} \\
+& \; e_1 \mu_{x-y e_2} e_3 - e_1 \mu_{x+y e_2} e_3 + \mu_{x+y e_3} e_3 - \mu_{x-y e_3} e_3. \nonumber
\end{align}

\begin{defn}
\label{def:Dec10tt2}
The Borel sets of $\sigma_S(T) \cap  \CC_j^+$ will be denoted by $\mathfrak{B}(\sigma_S(T)  \cap  \CC_j^+)$.
\end{defn}

\begin{lem}
\label{lem:Sept12yry1}
The $\mathbb{H}$-valued measure $\mu_{x,y}$ given in \eqref{eq:Nov29u2} enjoys the following properties{\rm :}
\begin{enumerate}
\item[(i)] $\mu_{x \alpha + y \beta, z} = \mu_{x,z} \alpha + \mu_{y,z} \beta, \quad \alpha, \beta \in \mathbb{H}$.

\smallskip

\item[(ii)] $\mu_{x, y \alpha + z \beta} = \bar{\alpha} \mu_{x,y}  + \bar{\beta} \mu_{x,z}, \quad \alpha, \beta \in \HH$.

\smallskip

\item[(iii)] $|\mu_{x,y} (\sigma_S(T) \cap  \CC_j^+)| \leq \| x \| \| y \|$.

\smallskip

\item[(iv)] $\bar{\mu}_{x,y} = \mu_{y,x}$,

\end{enumerate}
for all $x,y,z \in \cH$.
\end{lem}

\begin{proof}
Properties (i)-(iii) are easily obtained from \eqref{eq:Sept18i1} using the uniqueness of $\mu_{x,y}$ (relative to a fixed $j \in \mathbb{S}$) and the properties of $\langle \cdot, \cdot \rangle$. Property (iv) follows from properties (i) and (ii).
\end{proof}

It follows from Properties (i)-(iii) in Lemma \ref{lem:Sept12yry1} that $\Phi(x) = \mu_{x,y}(\sigma)$, where $y \in \mathcal{H}$ and $\sigma \in \mathfrak{B}(\sigma_S(T) \cap  \CC_j^+)$ are fixed,
is a continuous right linear functional on $\cH$. It follows from an analog of the Riesz representation theorem for Hilbert spaces (see Theorem 6.1 in \cite{AlpayShapiro}) that corresponding to any $x \in \cH$, there exists a unique vector $w \in \cH$ such that
$$\Phi(x) = \langle x, w \rangle,$$
i.e., $\mu_{x,y}(\sigma) = \langle x, w \rangle$. Using (i) and (ii) in Theorem \ref{lem:Sept12yry1}, we get the existence of an operator $E \in \cB(\cH)$ so that $w = E(\sigma)^* y$. Thus,
\begin{equation}
\label{eq:Sept12urr1}
\mu_{x,y}(\sigma) = \langle E(\sigma)x, y \rangle, \quad \sigma \in \mathfrak{B}(\sigma_S(T)  \cap  \CC_j^+).
\end{equation}
In view of \eqref{eq:Sept18i1} and \eqref{eq:Sept12urr1}, we may write
\begin{equation}
\label{eq:Sept18yb1}
g(T) = \int_{\sigma_S(T) \cap  \CC_j^+} g(p) dE(p), \quad g \in \mathscr{C}(\Omega_j^+, \RR).
\end{equation}

\begin{thm}
\label{thm:August25df1}
The $\cB(\cH)$-valued measure $E$, given by \eqref{eq:Sept18yb1}, enjoys the following properties{\rm :}
\begin{enumerate}
\item[(i)] $\| E(\sigma) \| \leq 1$.
\item[(ii)]  $E (\emptyset) = 0$ and $E(\sigma_S(T) \cap  \CC_j^+) = I_{\mathcal{H}}$.
\item[(iii)] If $\sigma \cap \tau = \emptyset$, then $E(\sigma \cup \tau) = E(\sigma) + E(\tau)$.
\item[(iv)] $E(\sigma \cap \tau) = E(\sigma)E(\tau)$.
\item[(v)] $E(\sigma)^* = E(\sigma)$.
\item[(vi)] $E(\sigma)^2 = E(\sigma)$.
\item[(vii)] $E(\sigma)$ commutes with $f(T)$ for all $f \in \mathscr{C}(\sigma_S(T) \cap  \CC_j^+, \CC_j)$.
\item[(viii)] $E(\sigma)$ and $E(\tau)$ commute for all $\sigma, \tau \in \mathfrak{B}(\sigma_S(T) \cap \CC_j^+)$.
\end{enumerate}
\end{thm}

\begin{proof}
Property (i) follows directly from property (iii) in Lemma \ref{lem:Sept12yry1}. Since $\mu_{x,y}(\emptyset) = 0$, we may use \eqref{eq:Sept12urr1} to deduce $E(\emptyset) = 0$. Similarly, putting $g(p) = 1$ in \eqref{eq:Sept18yb1} yields $g(T) = I_{\mathcal{H}}$ for all $x, y \in \cH$ and thus
$$
\langle x, y \rangle = \int_{\sigma_S(T) \, \cap \, \CC_j^+ } d\mu_{x,y} = \langle E(\sigma_S(T) \cap  \CC_j^+)x, y \rangle,
$$
i.e., $E(\sigma_S(T) \cap \CC_j^+ ) = I_{\mathcal{H}}$. Property (iv) follows easily from Property (i) of Lemma \ref{lem:Sept12yry1}. Property (v) follows easily from property (iv) in Lemma \ref{lem:Sept12yry1}.
Property (vi) can be obtained from Property (iv) when $\sigma = \tau$. In view of the fact that $\sigma \cap \tau = \tau \cap \sigma$, Property (viii) can be obtained from Property (iv).

We will now show that Property (vii) holds. It follows from Theorem \ref{thm:Sept1ub1} that $T = A + J B$, where $A \in \cB(\cH)$ is self-adjoint, $B \in \cB(\cH)$ is positive and $J \in \cB(\cH)$ is anti self-adjoint and unitary. Moreover, $A$, $B$ and $J$ all mutually commute. It follows from item (iii) in Theorem \ref{thm:July30e1} that $f(T)$ commutes with $J$, and moreover, by the construction given in Remark \ref{rem:Dec9t1}, $f(T)$ commutes with $A$ and $B$ for $f \in \cS_{\RR} (\sigma_S(T), \CC_j)$. In view of the identifications made in Lemma \ref{lem:Nov29q1}, we also have that $f(T)$ commutes with $A$ and $B$ for all $f \in \mathscr{C}(\sigma_S(T)\cap \CC_j^+, \CC_j)$, where $\mathscr{C}(\sigma_S(T) \cap \CC_j^+, \CC_j)$ denotes the set of $\CC_j$-valued continuous functions on $\sigma_S(T) \cap \CC_j^+$. To verify that $E(\sigma)$ commutes with $A$, note that
$$\langle f(T) Ax, y \rangle = \langle f(T) x, Ay \rangle \quad {\rm for} \quad x, y \in \cH.$$
Thus, in view of \eqref{eq:Sept22wu1} and \eqref{eq:Sept12urr1}, we have
$$\mu_{Ax, y} = \mu_{x, Ay}.$$
Consequently,
$$\langle E(\sigma) A x, y \rangle = \langle E(\sigma) x, Ay \rangle = \langle A E(\sigma) x, y \rangle.$$
Thus, $E(\sigma)$ and $A$ commute. In a similar fashion, one can show that $E(\sigma)$ and $B$ commute and also that $E(\sigma)$ and $J$ commute. Therefore, in view of \eqref{eq:Oct19j1}, \eqref{eq:Oct19j2} and \eqref{eq:Dec9rr2}, we have that $E(\sigma)$ and $f(T)$ commute.
\end{proof}

\begin{defn}
\label{def:August25uu1}
A $\cB(\cH)$-valued measure $E$ on $\sigma_S(T) \cap  \CC_j^+$ will be called a {\it spectral measure} if $E$ has Properties (i)-(viii) in Theorem \ref{thm:August25df1}. Note that $E(\sigma)$ is a positive operator for all $\sigma \in \mathfrak{B}(\sigma_S(T) \cap  \CC_j^+)$.
\end{defn}

\begin{rem}
\label{rem:Nov27w1}
Fix $j \in \mathbb{S}$ and let $\cN_j$ be a Hilbert basis of $\cH$ so that when $J$ is written with respect to $\cN_j$, we have $J = L_j$ (such a basis exists by item (v) of Lemma \ref{thm:Sept1ub1}). In Theorem \ref{thm:August25j1} we shall assume that $T$ is written with respect to the Hilbert basis $\cN_j$.
\end{rem}

In what follows we shall let $\mathscr{C}(\sigma_S(T) \cap  \CC_j^+, \CC_j)$ denote the set of $\CC_j$-valued continuous functions on $\sigma_S(T) \cap  \CC_j^+$. We are now ready to state and prove the main result of the section.

\begin{thm}
\label{thm:August25j1}
Let $T \in \cB(\cH)$ be normal and fix $j \in \mathbb{S}$. If $T$ is written with respect to the Hilbert basis $\cN_j$ {\rm (}see Remark \ref{rem:Nov27w1}{\rm )}, then there exists a unique spectral measure $E$ so that
\begin{equation}
\label{eq:Sept22wt1}
f(T) = \int_{\sigma_S(T) \, \cap \, \CC_j^+} f(p) \; dE(p),   \quad  f \in \mathscr{C}(\sigma_S(T) \cap  \CC_j^+, \CC_j).
\end{equation}
Moreover, $W \in \cB(\cH)$  commutes with $A$, $B$ and $J$, which appear in the decomposition $T= A + B J$ {\rm (}see Theorem \ref{thm:Sept1ub1}{\rm )}, if and only if $W$ commutes with $E(\sigma)$, $\sigma \in \mathfrak{B}(\sigma_S(T) \cap  \CC_j^+)$.
\end{thm}

\begin{proof}
If $f \in \mathscr{C}(\sigma_S(T) \cap  \CC_j^+, \CC_j)$, then $f = f_0 + j f_1$, where $f_0$ and $f_1$ both belong to $\mathscr{C}(\sigma_S(T) \cap \CC_j^+, \RR)$. In view of Lemma \ref{lem:Nov29q1}, $f$ can be identified with a function in $\cS_{\RR}(\sigma_S(T),\CC_j)$ (with a slight abuse of notation we will use $f$ for the function in $\cS_{\RR}(\sigma_S(T), \CC_j)$). Next, we have already observed in \eqref{eq:Dec9rr2} that
$$f(T) = f_0(T) + J f_1(T).$$
We will assume that $f(T)$ is written with respect to the Hilbert basis $\cN_j$ of $\cH$ so that $J = L_j$. Thus,
$$f(T) = f_0(T) + j f_1(T).$$
Thus, we can use \eqref{eq:Sept18yb1} on $f_0(T)$ and $f_1(T)$ to obtain a spectral measure $E$ on $\mathfrak{B}(\sigma_S(T) \cap \CC_j^+)$ so that
$$f_0(T) = \int_{\sigma_S(T) \, \cap \, \CC_j^+} f_0(p) dE(p),$$
$$f_1(T) = \int_{\sigma_S(T) \, \cap \, \CC_j^+} f_1(p) dE(p)$$
and, finally,
\begin{equation}
\label{eq:Sept22wu1}
f(T) = \int_{\sigma_S(T) \, \cap \, \CC_j^+} f(p) dE(p), \quad f \in \mathscr{C}(\sigma_S(T) \cap  \CC_j^+, \CC_j).
\end{equation}

To see that the spectral measure $E$ is unique for a fixed $j \in \mathbb{S}$, suppose not, i.e., there exists another spectral measure $\widetilde{E}$ on $\sigma_S(T) \cap \CC_j^+$ with $\tilde{\mu}_{x,y}(\sigma) = \langle \widetilde{E}(\sigma) x, y \rangle$, so that
$$f(T) = \int_{\sigma_S(T) \, \cap \, \CC_j^+} f(p) \, dE(p) = \int_{\sigma_S(T) \, \cap \, \CC_j^+} f(p) \, d\widetilde{E}(p)$$
for $f \in \mathscr{C}(\sigma_S(T) \cap  \CC_j^+, \CC_j)$. If we let
\begin{align*}
\Lambda(f) =& \;  \langle f(T)x, x \rangle, \quad f \in \mathscr{C}(\sigma_S(T) \cap \CC^+_j, \CC_j), \\
=& \; \int_{\sigma_S(T) \, \cap \, \CC_j^+} f(p) \, d\mu_{x,x}(p),
\end{align*}
then it is readily checked that $\Lambda$ is a positive bounded $\CC_j$-linear functional on $\mathscr{C}(\sigma_S(T) \cap  \CC_j^+, \CC_j)$. It follows from the uniqueness assertion in the Riesz representation for complex-valued linear functionals, see, e.g., Theorem 6.19 in \cite{Rudin}, that  $\mu_{x,x} = \tilde{\mu}_{x,x}$. Using the polarization formula \eqref{eq:Sept24yy1} we have
$$\mu_{x,y} = \tilde{\mu}_{x,y}, \quad x, y \in \cH,$$
and hence $E(\sigma) = \widetilde{E}(\sigma)$ for $\sigma \in \mathfrak{B}(\sigma_S(T) \cap  \CC_j^+)$.

We will now prove the last assertion. If $W \in \cB(\cH)$ commutes with $A$, $B$ and $J$, then $W$ commutes with $\psi(T)$ for any $\psi \in \mathscr{C}(\sigma_S(T)  \cap  \CC_j^+, \RR)$. In view of \eqref{eq:Sept22wt1},
$$\langle \psi(T) Wx, y \rangle = \int_{\sigma_S(T) \, \cap \, \CC_j^+} \psi(p) d\langle E(p) Wx, y \rangle$$
and
$$\langle  \psi(T)x, W^* y \rangle = \int_{\sigma_S(T) \, \cap \, \CC_j^+} \psi(p) d\langle E(p) x, W^* y \rangle.$$
Thus, as
$$\langle \psi(T) Wx, y\rangle = \langle \psi(T) x, W^* y\rangle$$
we have
$$\langle E(\sigma) x, W^* y \rangle = \langle E(\sigma) Wx, y \rangle$$
and, consequently,
$$\langle W E(\sigma)x, y \rangle =\langle E(\sigma) Wx, y \rangle.$$
Therefore, $W E(\sigma) = E(\sigma) W$ for all $\sigma \in \mathfrak{B}(\sigma_S(T)  \cap  \CC_j^+)$.

Conversely, suppose $W \in \cB(\cH)$ and $E(\sigma)$ commute for all $\sigma \in \mathfrak{B}(\sigma_S(T)  \cap  \CC_j^+)$. If $f(u+vj) = u$, $g(u+vj) = v$ and $h(u+vj) = j$, then $f$, $g$ and $h$ belong to $\mathscr{C}(\sigma_S(T)\cap \CC_j^+, \CC_j)$ and $f(T) = A$, $g(T) = B$ and $h(T) = J$. Consequently, using \eqref{eq:Sept22wt1}, we have
\begin{align*}
\langle AWx, y \rangle =& \; \int_{\sigma_S(T) \, \cap \, \CC_j^+} u \, d \langle E(p)Wx, y \rangle \\
=& \; \int_{\sigma_S(T) \, \cap \, \CC_j^+} u \, d \langle E(p)x W^*y \rangle \\
=& \; \langle Ax, W^* y \rangle \\
=& \; \langle WA x, y \rangle, \quad x, y \in \cH.
\end{align*}
Thus, $W$ and $A$ commute. The proof that $W$ commutes with $B$ and $J$ is completed in a similar fashion.
\end{proof}

\begin{cor}
\label{cor:Nov30ye1}
In the setting of Theorem {\rm \ref{thm:August25j1}}, the following statements hold{\rm :}
\begin{enumerate}
\item[(i)] If $T \in \cB(\cH)$ is a positive operator, then there exists a unique positive operator $T^{1/2} := W \in \cB(\cH)$ so that $W^2 = T$.
\item[(ii)] $T \in \cB(\cH)$ is self-adjoint if and only if
\begin{equation}
\label{eq:Nov30k1}
T = \int_{[-\| T \|, \| T \|]} t \, dE(t).
\end{equation}
\item[(iii)] $T \in \cB(\cH)$ is anti self-adjoint if and only if
\begin{equation}
\label{eq:Nov30h2}
T = \int_{[0, \| T \|]} j t \, dE(t).
\end{equation}
\item[(iv)] $T \in \cB(\cH)$ is unitary if and only if
\begin{equation}
\label{eq:Nov30hkh2}
T = \int_{[0, \pi]} e^{j t}  \, dE(t).
\end{equation}
\end{enumerate}
\end{cor}

\begin{proof}
As we have observed in item (i) of Theorem \ref{thm:Sept15yz1}, if $T \in \cB(\cH)$ is a positive operator, then $\sigma_S(T) \subseteq [0, \| T \|]$. Thus, using Theorem \ref{thm:August25j1} we have the existence of a uniquely determined spectral measure $E$ so that
\begin{equation}
\label{eq:Dec1yuu2}
T = \int_{[0, \| T \|]} t \, dE(t).
\end{equation}
Let $g(t) = t^{1/2}$ for $t \in \RR$. Since $g \in \mathscr{C}(\sigma_S(T), \RR)$, it follows from Theorem \ref{thm:August25j1} that
$$W := g(T) = \int_{[0, \| T \|]} t^{1/2} \, dE(t)$$
satisfies $W^2 = T$. Thus, we have established the existence of a positive operator $W \in \cB(\cH)$ so that $W^2 = T$. The proof that $W$ is unique follows from the uniqueness of the spectral measure $E$, just as in the case that $\cH$ is a complex Hilbert space.

The proofs of (ii)-(iv) follow readily from Theorem \ref{thm:August25j1} and \eqref{eq:Dec5tt1}.
\end{proof}

\section{Spectral integrals}
\label{sec:SIs}
The goal of this section is to extend the integral representation in Theorem \ref{thm:August25j1} to a more general class of functions. This will be useful when proving a spectral theorem for unbounded operators in Section \ref{sec:STU}.  To this end we will follow Chapter 4 of the book \cite{Schmuedgen}. Most of the proofs of the properties of spectral integrals are easily adapted from the classical case presented in \cite{Schmuedgen}, i.e., when $\cH$ is a complex Hilbert space. However, some facts require additional arguments which we will highlight.

Fix a normal operator $T \in \cB(\cH)$ and $j \in \mathbb{S}$. As in Section \ref{sec:STB}, we will assume throughout this section that $T$ is written with respect to a Hilbert basis $\cN_j$ of $\cH$ so that $J = L_j$. The existence of such a Hilbert basis is guaranteed by item (v) of Theorem \ref{thm:Sept1ub1}.

In view of Theorem \ref{thm:August25j1} we have the existence of a uniquely determined spectral measure $E$ so that \eqref{eq:Sept22wt1} holds.

\begin{defn}
\label{def:Dec15tt2}
Let $\mathscr{B}(E, \Omega_j^+, \CC_j)$, where $\Omega_j^+ = \sigma_S(T) \cap  \CC_j^+$, denote the set of all bounded $E$-measurable $\CC_j$-valued functions $f$ on $\Omega_j^+$ with the norm
$$\| f \|_{\infty} = \sup_{p \, \in \, \Omega_j^+} |f(p)|.$$
Let $\mathscr{B}_s(E, \Omega_j^+, \CC_j)$ denote the subset of simple functions in $\mathscr{B}(E, \Omega_j^+, \CC_j)$, i.e., all $\CC_j$-valued functions of the form $f(p) = \sum_{n=1}^k c_n \chi_{\sigma_n}(p)$, where $\sigma_1, \ldots \sigma_n$ are pairwise disjoint sets in $\mathfrak{B}(\Omega_j^+)$, $c_1, \ldots, c_n \in \CC_j$ and
$$\chi_{\sigma}(p) = \begin{cases} 1 & {\rm if} \quad p \in \sigma \\ 0 & {\rm if} \quad p \notin \sigma. \end{cases}$$
If $f \in \mathscr{B}_s(E, \Omega_j^+, \CC_j)$, then we may define
\begin{equation}
\label{eq:Nov29td2}
f(T) = \sum_{n=1}^k c_n E(\sigma_n).
\end{equation}
\end{defn}

\begin{lem}
\label{lem:Nov29jh1}
If $f \in \mathscr{B}_s(E, \Omega_j^+, \CC_j)$, then
\begin{equation}
\label{eq:Nov29re1}
\| f(T) \| \leq \| f \|_{\infty}.
\end{equation}
\end{lem}

\begin{proof}
If $f = \sum_{n=1}^k c_n \chi_{\sigma_n}$, where $c_1, \ldots, c_k \in \CC_j$ and $\sigma_1, \ldots, \sigma_k$ are disjoint sets in $\mathfrak{B}(\Omega_j^+)$, then
\begin{align*}
\| f(T) x \|^2 =& \; \| \sum_{n=1}^k c_n E(\sigma_n) x \|^2 \\
=& \; \sum_{n=1}^k |c_n|^2 \| E(\sigma_n) x \|^2 \\
\leq& \; \| f \|_{\infty}^2 \| x \|^2.
\end{align*}
Thus, \eqref{eq:Nov29re1} holds.
\end{proof}

Fix $f \in \mathscr{B}(E, \Omega_j^+, \CC_j)$. Since $\mathscr{B}_s(E, \Omega_j^+, \CC_j)$ is a dense subset of $\mathscr{B}(E, \Omega_j^+, \CC_j)$, there exists a sequence of functions $\{ f_n \}_{n=0}^{\infty}$ belonging to $\mathscr{B}_s(E, \Omega_j^+, \CC_j)$ so that
$$\lim_{n \uparrow \infty} \| f_n - f \|_{\infty} = 0.$$
In view of \eqref{eq:Nov29re1}, $\{ f_n(T) x \}_{n=0}^{\infty}$ is a Cauchy sequence in $\cH$. Let $f(T)$ be given by
$$f(T) x = \lim_{n \uparrow \infty} f_n(T) x, \quad x \in \cH.$$
Note that $f$ does not depend on the choice of the sequence $\{ f_n \}_{n=0}^{\infty}$ and, consequently, neither does $f(T)$.

\begin{lem}
\label{lem:Nov29w1}
If $f, g \in \mathscr{B}(E, \Omega_j^+, \CC_j)$, where $\Omega_j^+ = \sigma_S(T) \cap  \CC_j^+$, $\alpha, \beta \in \CC_j$ and $x, y \in \cH$, then{\rm :}
\begin{enumerate}
\item[(i)] $f(T)^* = \bar{f}(T)$, $(\alpha f + \beta g)(T) = \alpha f(T) + \beta g(T)$.
\item[(ii)] $\langle f(T)x, y \rangle = \int_{\Omega_j^+} f(t) d\langle E(p)x, y \rangle$.
\item[(iii)] $\| f(T) x \|^2 = \int_{\Omega_j^+} |f(p)|^2 d\langle E(p)x, x \rangle$.
\item[(iv)] $\| f(T) \| \leq \| f \|_{\infty}$.
\item[(v)] $(fg)(T) = f(T) g(T)$.
\end{enumerate}
\end{lem}

\begin{proof}
In view of the density of $\mathscr{B}_s(E, \Omega_j^+, \CC_j)$ in $\mathscr{B}(E, \Omega_j^+, \CC_j)$ and \eqref{eq:Nov29re1}, it suffices to check (i)-(v) when $f, g \in \mathscr{B}_s(E, \Omega_j^+, \CC_j)$. Assertions (i)-(iv) are checked in a straightforward manner. If
$f = \sum_{n=1}^k c_n \chi_{\sigma_n}$ and $g = \sum_{n=1}^k d_n \chi_{\tau_n}$ belong to $\mathscr{B}_s(E, \Omega_j^+, \CC_j)$, then
\begin{align}
(fg)(T) =& \; \sum_{m,n=1}^k c_m d_n E(\sigma_m \cap \tau_n ) \nonumber \\
=& \; \sum_{m,n=1}^k c_m d_n E(\sigma_m) E(\tau_n) \nonumber \\
=& \; \left\{ \sum_{m=1}^k c_m E(\sigma_m) \right\} \left\{ \sum_{n=1}^k d_n E(\tau_n) \right\}  \label{eq:Nov29cc1} \\
=& \; f(T) g(T). \nonumber
\end{align}
Line \eqref{eq:Nov29cc1} is justified by the fact that $J = L_j$ and, since $J$ commutes with $T$, $J$ must commute with $E(\sigma)$ for any $\sigma \in \mathfrak{B}(\Omega_j^+)$ (see the last assertion of Theorem \ref{thm:August25j1}). Consequently, $c E(\sigma) = E(\sigma) c$ for any $c \in \CC_j$ and $\sigma \in \mathfrak{B}(\Omega_j^+)$.
\end{proof}

We will now extend $\mathscr{B}(E, \Omega_j^+, \CC_j)$ to a more general class which will be useful when proving the spectral theorem for unbounded normal operators.

\begin{defn}
\label{def:Dec15yu3}
Let $\mathscr{B}_{\infty}(E, \Omega_j^+, \CC_j)$ denote the space of all $E$-measurable functions $f: \Omega_j^+ \to \CC_j \cup \{ \infty\}$ which satisfy
$$E( \{ p \in \Omega_j^+: f(p) = \infty \}) = 0.$$
\end{defn}

\begin{defn}
\label{def:Nov29u1}
A sequence of sets $\{ \sigma_n \}_{n=0}^{\infty}$, where $\sigma_n \in \mathfrak{B}(\Omega_j^+)$ for $n=0,1,\ldots$ is called a {\it bounding sequence} for a subset of functions $\mathfrak{F} \subseteq \mathscr{B}_{\infty}(E, \Omega_j^+, \CC_j)$ if
\begin{enumerate}
\item[(i)] For any $n=0,1,\ldots$, $f$ is bounded on $\sigma_n$.
\item[(ii)] $\sigma_{n} \subseteq \sigma_{n+1}$ for $n=0,1,\ldots$.
\item[(iii)] $E(\cup_{n=0}^{\infty} \sigma_n) = I_{\cH}$.
\end{enumerate}
\end{defn}

\begin{rem}
\label{rem:Nov29rr1}
If $\{ \sigma_n \}_{n=0}^{\infty}$ is a bounded sequence, then the following assertions follow from Theorem \ref{thm:August25df1}{\rm :}
\begin{enumerate}
\item[(i)] $E(\sigma_n) \preceq E(\sigma_{n+1})$.
\item[(ii)] $E(\sigma_n)x \to x$ as $n \uparrow \infty$, $x \in \cH$.
\item[(iii)] The set $\bigcup_{n=0}^{\infty} E(\sigma_n) \cH$ is dense in $\cH$.
\end{enumerate}
\end{rem}

We will now give meaning to $f(T)$ for $f \in \mathscr{B}_{\infty}(E, \Omega_j^+, \CC_j)$.

\begin{defn}
\label{def:Dec9tt2}
Fix $f \in \mathscr{B}_{\infty}(E, \Omega_j^+, \CC_j)$ and let $\{ \sigma_n \}_{n=0}^{\infty}$ be a bounding sequence for $f$. We let
\begin{equation}
\label{eq:Nov30u1}
f(T)x = \lim_{n \uparrow \infty} (\chi_{\sigma_n} f)(T)x
\end{equation}
with domain
\begin{equation}
\label{eq:Nov30u2}
\cD(f(T)) = \{ x \in \cH: \int_{\Omega_j^+} |f(p)|^2 d\langle E(p)x, x \rangle < \infty\}.
\end{equation}
\end{defn}

\begin{lem}
\label{lem:Nov30k1}
If $f \in \mathscr{B}_{\infty}(E, \Omega_j^+, \CC_j)$ and $\{ \sigma_n \}_{n=0}^{\infty}$ is a bounding sequence for $f$, then{\rm :}
\begin{enumerate}
\item[(i)] A vector $x$ belongs to $\cD(f(T))$ if and only if the sequence $\{ (\chi_{\sigma_n}f)(T)x \}_{n=0}^{\infty}$ converges in $\cH$, or, equivalently,
$$\sup_{n =0,1,\ldots } \| (f \chi_{\sigma_n})(T) x \| < \infty.$$
\item[(ii)] $f(T)$ does not depend on the choice of bounding sequence for $f$.
\item[(iii)] The set $\cup_{n=0}^{\infty} E(\sigma_n) \cH$ is a dense subset of $\cD(f(T))$. Moreover,
\begin{equation}
\label{eq:Nov30e1}
E(\sigma_n) f(T) \subseteq f(T) E(\sigma_n) = (f \chi_{\sigma_n})(T), \quad n =0,1,\ldots.
\end{equation}
\end{enumerate}
\end{lem}

\begin{proof}
We have already observed in Definition \ref{def:August25uu1} that $E(\sigma)$ is a positive operator on $\Omega_j^+$ for every $\sigma \in \mathfrak{B}(\Omega_j^+)$. Thus, $\mu_x$ is a positive measure on $\Omega_j^+$, where $\mu_x(\sigma)  = \langle E(\sigma)x, x \rangle$. Consequently, the proof of items (i)-(iii) can be completed in much the same way as in items (i)-(iii) of Theorem 4.13 in \cite{Schmuedgen}.
\end{proof}

In the following theorem, $\overline{W}$ shall denote the closure of an operator $W \in \cL(\cH)$.

\begin{thm}
\label{thm:Nov30k2}
If $f,g \in \mathscr{B}_{\infty}(E, \Omega_j^+, \CC_j)$ and $\alpha, \beta \in \CC_j$, then{\rm :}
\begin{enumerate}
\item[(i)] $\bar{f}(T) = f(T)^*$.
\item[(ii)] $(\alpha f + \beta g)(T) = \overline{\alpha f(T) + \beta g(T) }$.
\item[(iii)] $(fg)(T) = \overline{f(T)g(T)}$.
\item[(iv)] $f(T)$ is a closed normal operator on $\cH$ and
$$f(T)^*f(T) = (f \bar{f})(T) = (\bar{f} f )(T).$$
\item[(v)] $\cD( f(T) g(T) ) = \cD( g(T) ) \cap \cD( (fg)(T) )$.
\end{enumerate}
\end{thm}

\begin{proof}
The proof of items (i)-(iv) when $\cH$ is a complex Hilbert space (see items (i)-(v) of Theorem 4.16 in \cite{Schmuedgen}) can easily be adapted to the case when $\cH$ is a quaternionic Hilbert space.
\end{proof}

\begin{thm}
\label{thm:Nov30yy1}
If $f \in \mathscr{B}_{\infty}(E, \Omega_j^+, \CC_j)$, then $f(T)$ is invertible if and only if $f$ does not vanish $E$-a.e. on $\Omega_j^+$. In this case,
\begin{equation}
\label{eq:Nov30rew2}
f(T)^{-1} = (1/f)(T),
\end{equation}
where we use the convention that $1/0 = \infty$ and $1/\infty = 0$.
\end{thm}

\begin{proof}
The proof when $\cH$ is a complex Hilbert space (see Proposition 4.19 in \cite{Schmuedgen}) can easily be adapted to the case when $\cH$ is a quaternionic Hilbert space.
\end{proof}

\begin{lem}
\label{lem:Nov30gh2}
If $g: \Omega_j^+ \to \widetilde{\Omega}_j^+ \subseteq \CC_j$, $h \in \mathscr{B}_{\infty}(\widetilde{E}, \widetilde{\Omega}_j^+, \CC_j)$ and $\widetilde{E}$ is the spectral measure on $\tilde{\Omega}_j^+$ given by
$$\widetilde{E}(\sigma) = E(g^{-1}(\sigma)), \quad \sigma \in \mathfrak{B}(\widetilde{\Omega}_j^+),$$
then $h \circ g \in \mathscr{B}_{\infty}(E, \Omega_j^+, \CC_j)$ and
\begin{equation}
\label{eq:Nov30yh2}
\int_{\widetilde{\Omega}_j} h(p) d\widetilde{E}(p) = \int_{ {\Omega}_j^+ } h(g(p)) dE(p).
\end{equation}
\end{lem}

\begin{proof}
 The proof is similar to the proof when $\cH$ is a complex Hilbert space (see Proposition 4.24 in \cite{Schmuedgen}) except for the fact that the polarization formula \eqref{eq:Sept24yy1} needs to be used.
\end{proof}

\section{The spectral theorem for unbounded normal operators based on the $S$-spectrum}
\label{sec:STU}
In this section we will consider normal operators $T$ which are unbounded, i.e., $T \in \cL(\cH)$ but $T \notin \cB(\cH)$. The strategy will be to transform $T$ into a normal operator $Z_T \in \cB(\cH)$ and use Theorem \ref{thm:August25j1} and a change of variable argument to obtain a spectral theorem for $T$ based on the $S$-spectrum. Obtaining a spectral theorem for unbounded operators in the aforementioned way has been done in the classical case, i.e., when $\cH$ is a complex Hilbert space. See, e.g., the book of Schm\"udgen \cite{Schmuedgen}.

Given $T \in \cL(\cH)$, we let
\begin{equation}
\label{eq:August25nv1}
Z_T = T C_T^{1/2},
\end{equation}
where $C_T = ( I_{\cH} + T^* T)^{-1}$. Note that $C_T$ is a bounded positive operator on $\cH$ and
\begin{equation}
\label{eq:August25hn1}
C_T = I_{\cH} - Z_T^* Z_T.
\end{equation}
Formula \eqref{eq:August25hn1} follows from
\begin{align*}
I_{\cH} - Z_T^* Z_T =& \; I - C_T^{1/2} T^* T C_T^{1/2} = C_T^{1/2} ( C_T^{-1} - T^* T ) C_T^{1/2} \\
=& \; C_T.
\end{align*}

\begin{thm}
\label{thm:August25er5}
Let $T \in \cL(\cH)$ be a densely defined closed operator on $\cH$. The operator $Z_T$ has the following properties{\rm :}
\begin{enumerate}
\item[(i)] $Z_T \in \cB(\cH)$, $\| Z_T \| \leq 1$ and
\begin{equation}
\label{eq:Sept3yy1}
C_T = (I_{\cH} + T^* T)^{-1} = I_{\cH} - Z_T^* Z_T.
\end{equation}
\item[(ii)] $(Z_T)^* = Z_{T^*}$.
\item[(iii)] If $T$ is normal, then $Z_T$ is normal.
\end{enumerate}
\end{thm}

\begin{proof} The proof is based on the proof of Lemma 5.7 in \cite{Schmuedgen} and is broken into steps.

\bigskip

\noindent {\bf Step 1:} {\it Prove {\rm (i)}}. \\

\bigskip

First note that
\begin{equation}
\label{eq:Sept3yt1}
\{\text{$C_T x$: $x \in \cH$}\} = \cD(I_{\cH} + T^* T ) = \cD(T^* T).
\end{equation}
Consequently, if $x \in \cH$, then
\begin{align*}
\| T C_T^{1/2} C_T^{1/2} x \|^2 =& \; \langle T^* T C_T x, C_T x \rangle \\
\leq& \; \langle (I_{\cH} + T^* T) C_T x, C_T x \rangle \\
=& \; \langle C_T^{-1} C_T x , C_T x \rangle \\
=& \; \langle x, C_Tx \rangle \\
=& \; \| C_T^{1/2} x \|^2.
\end{align*}
Thus, if $y \in \{\text{$C_T^{1/2}x$: $x \in \cH$}\}$, then
\begin{equation}
\label{eq:Sept3ttu1}
\| Z_T y \| = \| T C_T^{1/2} y \| \leq \| y \|.
\end{equation}
As ${\rm Ker}(C_T) = \{ 0 \}$, we have that ${\rm Ker}(C_T^{1/2}) = \{ 0 \}$ and thus $\{\text{$C_T^{1/2} x$: $x \in \cH$}\}$ is a dense subset of $\cH$. As $T$ is a closed operator by assumption and $C_T^{1/2} \in \cB(\cH)$, we get that $Z_T$ is closed as well. Thus, we have $\{\text{$C_T^{1/2}x$: $x \in \cH$}\} \subseteq \cD(T)$, $\cD(Z_T) = \cH$ and, in view of \eqref{eq:Sept3ttu1}, $\| Z_T \| \leq 1$.

Next, it follows from \eqref{eq:Sept3ttu1} and $C^{1/2} T^* \subseteq Z_T^*$ that
\begin{align*}
(I_{\cH} - C_T) C_T^{1/2} =& \; C_T^{1/2} (I_{\cH} + T^* T) C_T - C_T^{1/2} C_T \\
=& \; C_T^{1/2} T^* T C_T^{1/2} C_T^{1/2} \\
\subseteq & \; Z_T^* Z_T C_T^{1/2}.
\end{align*}
Thus, $Z_T^* Z_T C_T^{1/2} = (I_{\cH} - C_T )C_T^{1/2}$ and, as $\{\text{$C_T^{1/2} x$: $x \in \cH$}\}$ is a dense subset of $\cH$, we get \eqref{eq:Sept3yy1}.

\bigskip

\noindent {\bf Step 2:} {\it Prove {\rm (ii)}}. \\

\bigskip

Using \eqref{eq:Sept3yy1} we get that $C_{T^*} = (I_{\cH} + T T^*)^{-1}$. If $x \in \cD(T^*)$, then let $y = C_{T^*} x$. Therefore,
$$x = (I_{\cH} + T T^* )y $$
and
$$T^* x = T^* (I_{\cH} + TT^*)y = (I_{\cH} + T^* T) T^* y.$$
Thus, $C_{T^*} x \in \cD(T^*)$ and hence
\begin{equation}
\label{eq:Sept3rtz1}
C_T T^* x = T^* y = T^* C_{T^*} x.
\end{equation}
It follows easily from \eqref{eq:Sept3rtz1} and \eqref{eq:Sept3yy1} that $p(C_{T^*})x \in \cD(T^*)$ and
$$p (C_T ) T^* x = T^* p(C_{T^*} ) x$$
for any real polynomial $p$ of a real variable. By the Weierstrass approximation theorem, there exists a sequence of real polynomials $\{ \phi_n \}_{n=0}^{\infty}$ which converge uniformly in supremum norm to the function $t \mapsto t^{1/2}$ on $[0, 1]$. Using Property (iv) of Theorem \ref{thm:July30e1} we have that
$$\lim_{n \uparrow \infty} \| \phi_n(C_T) - C_T^{1/2} \| = \lim_{n \uparrow \infty} \| \phi_n( C_{T^*}) - C_{T^*}^{1/2} \| = 0.$$

Since $T$ is a closed operator, $T^*$ is also a closed operator. Thus, we have
\begin{align*}
C_T^{1/2} T^* x =& \; \lim_{n \uparrow \infty} \phi_n(C_T) T^* x = \lim_{n \uparrow \infty} T^* \phi_n (C_{T^*}) x \\
=& \; T^* (C_{T^*})^{1/2} x \quad {\rm for} \quad x \in \cD(T^*).
\end{align*}
As $C_T^{1/2} T^* \subseteq (TC_T^{1/2})^* = Z_{T^*}$, we get that
$$Z_{T^*} x = C_T^{1/2} T^* x = T^* (C_{T^*})^{1/2} x = (Z_{T})^* x$$
for $x \in \cD(T^*)$. Finally, since $\cD(T^*)$ is dense in $\cH$, we have that $Z_{T^*} x = (Z_T)^* x$, i.e., $Z_{T^*} = (Z_T)^*$.

\bigskip

\noindent {\bf Step 3:} {\it Prove {\rm (iii)}.} \\

\bigskip

Using \eqref{eq:August25hn1} on $T$ and $T^*$ and the fact that $TT^* = T^* T$ we have
$$I_{\cH} - Z_T^* Z_T = (I_{\cH} + T^* T )^{-1} = (I_{\cH} + T T^* )^{-1} = I_{\cH} - Z_{T^*}^* Z_{T^*}.$$
Making use of Property (ii) we have that
$$I_{\cH} - Z_T^* Z_T = I_{\cH} - Z_T Z_T^*,$$
i.e., $Z_T$ is normal.
\end{proof}

We are now ready to state and prove a spectral theorem for unbounded normal operators on a quaternionic Hilbert space.

\begin{thm}
\label{thm:August27t1}
Fix $j \in \mathbb{S}$ and let $T$ be an unbounded right linear normal operator on $\cH$, i.e., $T \in \cL(\cH)$, $T \notin \cB(\cH)$ and $T$ is normal. There exists a Hilbert basis $\cN_j$ of $\cH$ and a unique spectral measure $E$ so that if $T$ is written with respect to $\cN_j$, then
\begin{equation}
\label{eq:August27re5}
T = \int_{\sigma_S(T) \, \cap \, \CC_j^+} p \, dE(p).
\end{equation}
\end{thm}

\begin{proof}
The proof is broken into steps.

\bigskip

\noindent {\bf Step 1:} {\it Show that a spectral measure $E$ exists so that \eqref{eq:August27re5} holds.}

\bigskip

Let $\BB = \{ p \in \HH: | p | < 1 \}$, $\partial \BB = \{ p \in \HH: | p | = 1 \}$ and $\overline{\BB} = \BB \cup \partial \BB$. If $T$ is normal, then using Properties (i) and (iii) in Theorem \ref{thm:August25er5} we get that $\| Z_T \| \leq 1$ and $Z_T$ is normal, respectively. Thus, we may use Theorem \ref{thm:August25j1} to obtain a Hilbert basis $\cN_j$ of $\cH$ so that $J = L_j$ and also a uniquely determined spectral measure $F$ on $\sigma_S(Z_T) \cap \CC_j^+$ so that
\begin{equation}
\label{eq:August27rh1}
f(Z_T) = \int_{\sigma_S(Z_T) \cap \CC_j^+} f(p) \, dF(p)  \quad {\rm for} \quad f \in \mathscr{C}(\sigma_S(Z_T) \cap \CC_j^+, \CC_j).
\end{equation}
In addition, it follows from Theorem 3.2.6 in \cite{CSS} that
$$\sigma_S(Z_T) \subseteq \{ p \in \HH: |p| \leq \| Z_T \| \}$$
and hence
$$\sigma_S(Z_T) \cap \CC_j^+ \subseteq \overline{\BB} \cap \CC_j^+.$$

If $x \in \cH$ and $\sigma \in \mathfrak{B}(\sigma_S(T)  \cap  \CC_j^+)$, then, in view of, item (v) in Lemma \ref{lem:Nov29w1} and (\ref{eq:August27rh1}), we have
\begin{equation}
\label{eq:August27ww1}
\langle ( I_{\cH} - Z_T^* Z_T ) F(\sigma) x, F(\sigma) x \rangle = \int_{\sigma} (  1 - |p|^2 ) d\langle F(p) x, x \rangle.
\end{equation}
Recall that $I_{\cH} - Z_T^* Z_T = (I_{\cH} + T^* T)^{-1}$, i.e., ${\rm Ker}(I_{\cH} - Z_T^* Z_T) = \{ 0 \}$. Thus, using \eqref{eq:August27ww1} with
$$\sigma = \partial \overline{\mathbb{B}} \cap \CC_j^+$$ we get that ${\rm supp}\;F \subseteq \overline{\BB} \cap \CC_j^+$ and $F( \partial \BB \cap \CC_j^+) = 0$. Therefore,
$$F(\BB \cap \CC_j^+) = F [ (\overline{\BB}\cap \CC_j^+) \setminus \partial (\BB\cap \CC_j^+) ] = I_{\cH}.$$

If $\varphi(p) = p(1 - |p|^2)^{-1/2} $, then $\varphi \in \mathscr{B}_{\infty}(F, \sigma_S(Z_T)\cap \CC_j^+, \CC_j)$. In view of item (iii) and (v) of Theorem \ref{thm:Nov30k2} we have
$$\varphi(Z_T) = \left( \int_{\overline{\BB} \cap \CC_j^+} p \, dF(p) \right) \left( \int_{\overline{\BB} \cap \CC_j^+} \frac{1}{\sqrt{1-|p|^2}} \, dF(p) \right)$$
and $\cD(\varphi(Z_T)) = \cD(g(Z_T))$, where $g \in \mathscr{B}_{\infty}(F, \sigma_S(Z_T) \cap \CC_j^+, \CC_j)$ is given by
$$g(p) = \frac{1}{\sqrt{1-|p|^2}}.$$
Using Theorem \ref{thm:Nov30yy1}, we have
$$g(T) = (1/g)(T)^{-1}.$$
Consequently, we may use item (i) in Corollary \ref{cor:Nov30ye1} to obtain
$$g(T) = \left\{ \left(\int_{\overline{\BB} \cap \CC_j^+} (1 - |p|^2 ) \, dF(p) \right)^{1/2} \right\}^{-1}.$$
Putting these observations together, we obtain
\begin{equation}
\label{eq:August27yr1}
\varphi(Z_T) = Z_T (C_T^{1/2})^{-1}.
\end{equation}

Since $Z_T = T C_T^{1/2}$ we obtain $\varphi(Z_T) \subseteq T$. Using $C_T = (I_{\cH} - Z_T^* Z_T )^{1/2}$, we get that $\varphi(Z_T) \subseteq T$. Thus, using Lemma \ref{lem:Aug25yq1} we get that
$$\varphi(Z_T) = T.$$
Let $E(\sigma) = F(\varphi^{-1}(\sigma))$, where
$$\varphi^{-1}(\sigma) = \{ p \in \HH: \varphi(p) \in \sigma \} \quad {\rm for} \quad \sigma \in \mathfrak{B}(\sigma_S(T)  \cap  \CC_j^+).$$
It is readily checked that $E = F(\varphi^{-1})$ defines a spectral measure on $\CC_j^+$ and thus using Lemma \ref{lem:Nov30gh2} we have
\begin{equation}
\label{eq:August27tt1}
T = \int_{\overline{\BB} \cap \CC_j^+} \varphi(p) \, dF(p)  = \int_{\sigma_S(T) \, \cap \, \CC_j^+} p \, dE(p).
\end{equation}

\bigskip

\noindent {\bf Step 2:} {\it Show that $E$ from Step 1 is unique.}

\bigskip

If $E$ and $\widetilde{E}$ are spectral measures on $\sigma_S(T) \cap \CC_j^+$ which satisfy \eqref{eq:August27re5}, then $F = E(\varphi)$ and $\widetilde{F} = \widetilde{E}(\varphi)$ are both spectral measures so that
\begin{equation}
\label{eq:August27yq1}
Z_T = \int_{\overline{\BB} \cap \CC_j^+} p \, dF(p)   = \int_{\overline{\BB} \cap \CC_j^+} p \, d\widetilde{F}(p).
\end{equation}
Let $\mathscr{P}(\sigma_S(T) \cap \CC_j^+, \CC_j)$ denote the space of $\CC_j$-valued polynomials $\psi(p) = \sum_{a=0}^b \psi_a p^a$ on $\sigma_S(T) \cap \CC_j^+$.
In view of Lemma \ref{lem:Nov29w1}, \eqref{eq:August27yq1} yields
\begin{align*}
\langle \psi(Z_T)x, x \rangle =& \; \int_{\sigma_S(T) \, \cap \, \CC_j^+} \psi(p) d \langle F(p)x, x \rangle \\
=& \; \int_{\sigma_S(T) \, \cap \, \CC_j^+} \psi(p) d \langle \widetilde{F}(p)x, x \rangle, \quad \psi \in \mathscr{P}(\sigma_S(T) \cap \CC_j^+, \CC_j).
\end{align*}
As $\mathscr{P}(\sigma_S(T) \cap \CC_j^+, \CC_j)$ is a dense subset of $\mathscr{C}(\sigma_S(T) \cap \CC_j^+, \CC_j)$ we have that
$$\int_{\sigma_S(T) \, \cap \, \CC_j^+} \psi(p) d \langle F(p)x, x \rangle = \int_{\sigma_S(T) \, \cap \, \CC_j^+} \psi(p) d \langle \widetilde{F}(p)x, x \rangle.$$
Finally, use the uniqueness in Theorem \ref{thm:August25j1} to obtain that $F = \widetilde{F}$ and hence $E = \widetilde{E}$.
\end{proof}

\begin{cor}
\label{cor:Dec1yt1}
In the setting of Theorem {\rm \ref{thm:August27t1}}, the following statements hold{\rm :}
\begin{enumerate}
\item[(i)] If $T \in \cL(\cH)$ is a positive operator, then there exists a unique positive operator $W \in \cL(\cH)$ so that $W^2 = T$.
\item[(ii)] $T \in \cL(\cH)$ is self-adjoint if and only if
\begin{equation}
\label{eq:Nov30kk1}
T = \int_{\RR} t \, dE(t).
\end{equation}
\item[(iii)] $T \in \cL(\cH)$ is anti self-adjoint if and only if
\begin{equation}
\label{eq:Nov30hh2}
T = \int_{[0, \infty)} j t \, dE(t).
\end{equation}
\end{enumerate}
\end{cor}

\begin{proof}
Using Theorem \ref{thm:August27t1}, the proof is completed as in Corollary \ref{cor:Nov30ye1}.
\end{proof}

\end{document}